\documentclass[11pt,a4paper,reqno]{amsart}
\usepackage{a4wide,amsmath,amsfonts,amssymb,amsthm,mathrsfs,
setspace,bm,amsxtra,mathtools,mathrsfs,setspace,wasysym,textcomp} 
\mathtoolsset{showonlyrefs}
\usepackage[english]{babel}
\usepackage{graphicx,tikz}   
\usepackage[all,arrow,graph]{xy}
\usepackage{pdflscape}
\usepackage{verbatim}
\newtheorem{thm}{Theorem}[section]
\newtheorem{prop}[thm]{Proposition}
\newtheorem{lemma}[thm]{Lemma}
\newtheorem{defin}[thm]{Definition}
\newtheorem{cor}[thm]{Corollary}
\newtheorem{rema}[thm]{Remark}
\newtheorem{exa}[thm]{Example}

\numberwithin{equation}{section}
\newenvironment{rem}{\begin{rema}\rm}{\parbox{2mm}{\hfill}\hfill $\triangle$\end{rema}}
\newenvironment{ex}{\begin{exa}\rm}{\parbox{2mm}{\hfill}\hfill $\triangle$\end{exa}}

\newcommand{\marginnote}[1]{\ifthenelse{\isodd{\thepage}}{\normalmarginpar}
{\reversemarginpar}\marginpar{\fbox{\parbox{18mm}{
\footnotesize #1}}}}
\newcommand{\coker}{\operatorname{coker}}

\newcommand{\Hom}{\operatorname{Hom}}
\newcommand{\End}{\operatorname{End}}

\newcommand{\Ol}{\mathcal{O}}

\newcommand{\R}{\mathcal{R}}
\newcommand{\A}{\mathcal{A}}
\newcommand{\B}{\mathcal{B}}
\newcommand{\G}{\mathcal{G}}
\newcommand{\F}{\mathcal{F}}

\newcommand{\E}{\mathcal{E}}
\newcommand{\U}{\mathcal{U}}
\newcommand{\V}{\mathcal{V}}
\newcommand{\W}{\mathcal{W}}
\newcommand{\K}{\mathcal{K}}
\newcommand{\T}{\mathcal{T}}

\newcommand{\Hg}{\mathcal{H}}
\newcommand{\C}{\mathcal{C}}

\newcommand{\Ml}{\mathscr{M}}

\newcommand{\id}{\operatorname{id}}

\newcommand{\lExt}{\operatorname{\mathcal{E}\mspace{-2mu}\textit{xt}}}

\newcommand{\lHom}{\operatorname{\mathcal{H}\!\textit{om}}}
\newcommand{\lEnd}{\operatorname{\mathcal{E}\mspace{-2mu}\textit{nd}}}

\newcommand{\Z}{\mathbb{Z}}

\newcommand{\Pu}{\mathbb{P}^1}

\newcommand{\Hl}{\mathscr{H}}
\newcommand{\Al}{\mathscr{A}}

\newcommand{\Ext}{\operatorname{Ext}}

\newcommand{\El}{\mathscr{E}}
\newcommand{\Hb}{\mathbb{H}}

\newcommand{\Kl}{\mathscr{K}}

\usepackage{color}
\definecolor{rosso}         {rgb}{1.00 , 0.00 , 0.00}

\newcommand{\modd}[1]{#1\mbox{-\bf mod}}
\newcommand{\Rep}[1]{\mbox{\bf Rep}(#1)}
\usepackage{hyperref}
\hypersetup{
    colorlinks=true, 
    linktoc=all,     
    linkcolor=blue,  
}

\title[Homology of twisted quiver  bundles with relations]{Homology of twisted quiver \\[8pt] bundles with relations}

\begin{document}
\date{Revised \today}
\subjclass[2010]{14D20, 16G20, 18G15,  55T99} 
\keywords{Quiver algebras, quiver representations, moduli space, Ext groups}
\thanks{E-mail: {\tt bartocci@dima.unige.it, bruzzo@sissa.it, clsrava@gmail.com}\\ \indent
Research partially supported by {\sc prin}   ``Geometria delle variet\`a  algebriche'', by {\sc gnsaga-in}d{\sc am}, and by the University of Genova through the research grant ``Aspetti matematici nello studio delle interazioni fondamentali''.  
 U.B. is a member of the {\sc vbac} group.}

\maketitle
\markright{\sc C. Bartocci, U. Bruzzo, C.L.S. Rava}
\thispagestyle{empty}

\begin{center}{\sc Claudio Bartocci,$^\P$ Ugo Bruzzo$^{\S\dag\ddag\sharp}$ and Claudio L. S. Rava$^\P$ } \\[10pt]  \small 
$^\P$Dipartimento di Matematica, Universit\`a di Genova, \\ Via Dodecaneso 35, 16146 Genova, Italia \\[3pt]
  $^{\S}$ SISSA (Scuola Internazionale Superiore di Studi Avanzati),\\   Via Bonomea 265, 34136 Trieste, Italia \\[3pt]
  $^\dag$ IGAP  (Institute for Geometry and Physics), Trieste \\[3pt]
  $^\ddag$ INFN (Istituto Nazionale di Fisica Nucleare), Sezione di Trieste \\[3pt] $^\sharp$ Arnold-Regge Center for Algebra, Geometry \\ and Theoretical Physics, Torino
\end{center}
\maketitle

\begin{abstract}  We study the Ext modules in the category of   left modules  over
a twisted algebra of a finite quiver over a ringed space $(X,\Ol_X)$, allowing for the presence of relations. We introduce a spectral sequence which relates the Ext modules in that category with the Ext modules
in the category of $\Ol_X$-modules. Contrary to what happens in the absence of relations, this
spectral sequence in general does not degenerate at the second page. 
We also consider local Ext sheaves. Under suitable hypotheses, the Ext modules are represented as hypercohomology groups.
 \end{abstract}

  \maketitle
 
 {\small \setcounter{tocdepth}{1}
 \tableofcontents}

\section{Introduction}

In 1977 Barth used monads (3-term complexes of vector bundles that have nonzero cohomology only at their middle term) to construct moduli spaces of rank 2, degree 0 vector bundles on the
projective plane $\mathbb P^2$. That construction was generalized by Barth and Hulek \cite{Hul79,BaHu78}, and in 1985 Dr{\'e}zet and Le Potier \cite{DLP} eventually built  the moduli space of all Gieseker-semistable torsion-free sheaves $\mathbb P^2$. Their work used in a critical way the fact that
the relevant monad is a Kronecker complex, and indeed the moduli space can be built
as a moduli space of semistable Kronecker complexes modulo the action of a reductive group. 
Now  semistable Kronecker complexes can be regarded as semistable representations of
a quiver   with relations, whose moduli spaces were constructed by King  \cite{KiQ}. 
Moreover, moduli spaces of semistable sheaves on $\mathbb P^1\times\mathbb P^1$ were
constructed using quiver moduli spaces in \cite{Mai17}.

Another approach to the use of quiver moduli spaces to construct moduli spaces of sheaves uses the theory of Bridgeland stability conditions \cite{Bri07}, see \cite{Ohk10,ArMi17}.

Quivers have been of particular importance in the study of framed sheaves, whose moduli spaces can be regarded as higher rank generalizations of Hilbert schemes of points, and for some spaces are resolutions of singularities of instanton moduli spaces \cite{Do,Na99}. The description of these spaces by means of the so-called ADHM data indeed often allows one to identify them as (components) of a moduli space of quiver representations \cite{Na94,kuz,bblr1}; the reader may refer to   \cite{BLR1} for a review of this aspect, and to \cite{Ginz} for an introduction to quiver varieties.

Other moduli spaces that have been treated by means of quiver techniques are Higgs bundles, stable pairs and triples, see e.g.~\cite{ACGP}.
  
So, once the significance of the moduli spaces of quiver representations  for moduli problems has been established, it becomes important to study their main properties, in the first instance their deformations. This requires the knowledge of the main homological properties of moduli   of quiver representations. In \cite{GoKi} Gothen and King studied the homological algebra of the category of representations of a twisted quiver  without relations over the category of $\Ol_{X}$-modules on a ringed space $(X,\Ol_{X})$. 
Their main result is a long exact sequence which relates the Ext groups in the category of   representations of the quiver algebra with the Ext groups in the category of $\Ol_X$-modules, thus making the former potentially computable. In this paper we generalize Gothen and King's work   allowing for the presence of relations. We also consider local Ext sheaves in addition to global Ext modules.

Let $Q$ be a quiver, and let $(X,\Ol_X)$ be a ringed space. Let $\B$ be the sheaf of $\Ol_X$-algebras generated by the vertices of the quiver, and let $\Ml$ be a sheaf of $\Ol_X$-modules, graded over the arrows of the quiver; it has a natural $\B$-module structure. The $\Ml$-twisted quiver algebra $\Ml Q$ of $Q$ is the tensor algebra of $\Ml$ over $\B$. The $\Ml$-twisted quiver algebra $Q$ with relations
will be a quotient $\Al = \Ml Q/\Kl$, where $\Kl$ is a sheaf of two-sided ideals (the ideal of relations).
We shall  denote by $\Rep{\Ml Q}$ the category of  $\Ml$-twisted representations of $Q$.
  
In section \ref{SecRep}, Theorem \ref{ThmRQmod}, we prove, under a technical assumption on the $\Ol_X$-module $\Ml$, that 
the category $\Rep{\Ml Q} $ is equivalent to the category $\modd{\Ml Q}$, thus recovering a result of \cite{ACGP, GoKi} with a mild generalization, due to our slightly weaker assumptions; in particular, $\modd{\Al}$ is a full subcategory of $\Rep{\Ml Q} $. It may be interesting to note that one has an equivalence of categories between a suitable subcategory $\Rep{\Al}$ of $\Rep{\Ml Q}$ and $\modd{\Al}$ even in the case with relations but with a trivial twist. Section \ref{untwisted} studies indeed the untwisted case.

In Section \ref{ss} we address the problem of relating the Ext modules in the category
 $\Al$-{\bf mod} with those in the category $\Ol_X$-{\bf mod}. While in \cite{GoKi}
 this leads to a long exact sequence, here we only find a spectral sequence relating the two Exts, which
 in general does not degenerate at the second step. 
 This happens because the spectral sequence is associated with a first quadrant double complex
 that in general is unbounded in both directions, which in turns depends from the fact that in our case the resolution \eqref{eq-C(V)} is infinite.
 
 Theorem \ref{thm-hyper} generalizes the analogous result from  \cite{GoKi}, namely, the Ext modules --- when their first argument is locally free as an $\Ol_X$-module ---  can be realized as hypercohomology groups. Under the same hypothesis, Corollary \ref{cor-spect-seq-V-loc-free} expresses this result in local form, i.e., for the local Ext sheaves. 
 
 Section \ref{adhm} describes an example  (ADHM sheaves) where a vanishing theorem for the Ext groups can be proved. 
 
 Natural developments of these first results would be a base change formula for the local Exts, 
 and a study of the deformation theory of these moduli spaces, which would allow one to characterize their tangent spaces and tangent sheaf. We shall address these issues in a future publication.
  
 After we wrote a first draft of this paper, the preprint \cite{Moz}  appeared, were some related results are given (there no relations are considered, but the notion of twisting is generalized).  
 
 \smallskip
\noindent{\bf Acknowledgements}. We thank V.~Lanza and F.~Sala for useful discussions. Moreover we thank SISSA and the Department of Mathematics of the University of Genoa for hospitality and support during cross visits among the authors. During the preparation of this manuscript U.B.~was visiting the Instituto de Matem\'atica e Estat\'istica of the University of S\~ao Paulo (IME-USP) supported by the FAPESP grant
 2017/22091-9.
 
 \bigskip
\section{Twisted quiver algebras with relations}

Let $Q=(I,E,h,t)$ be a quiver,  where the set $I$ labels the vertices of $Q$, $E$ labels the arrows, and the maps
$t,h \colon E \to I$ associate with any arrow its origin (tail) and destination (head). We will assume throughout this paper that $Q$ is finite, in the sense that both sets $I$ and $E$ are finite. For each $i\in I$, we denote by $e_i$ the trivial path starting and ending at the vertex $i$. The free Abelian group $\mathbf B$  generated by $ \{e_{i}\}_{i\in I}$ has a natural ring structure given by
\begin{equation}\label{BRelations} e_{i}e_{j}=\delta_{ij}e_{i} \qquad\text{for all}\quad i, j \in I\,.
\end{equation}
Note that  $\sum_{i\in I}e_{i}=1_{\mathbf B}$, so that the elements $\{e_{i}\}_{i\in I}$ are a complete set of orthogonal idempotents. 
We remark that, because of the relations \eqref{BRelations}, one has the identification
\begin{equation}\label{tensorvspath}
T_{\mathbf B} (\bigoplus_{\alpha\in E} \Z\alpha) = \Z Q\,,
\end{equation}
where $T_{\mathbf B}(-)$ is the tensor algebra functor on the category of $\mathbf B$-bimodules and $\Z Q$ is the path algebra of $Q$ over $\Z$.

Let $(X,\Ol_{X})$ be a ringed space.  The free $ \Ol_{X}$-module 
$$\B= \bigoplus_{i\in I}\Ol_{X}e_{i}$$
 generated by $ \{e_{i}\}_{i\in I}$
can be endowed with a structure of a sheaf of $\mathbf B$-algebras by imposing conditions that are formally the same as those in eq.~\eqref{BRelations}.

Consider now a collection $\{\Ml_{\alpha}\}_{\alpha\in E}$ of  $\Ol_{X}$-modules labeled by the arrows of  $Q$. Each 
$\Ml_{\alpha}$ can be endowed  with a structure of $\B$-bimodule in the following way: for every open subset $U\subseteq X$ and for every section $x\in\Ml_{\alpha}(U)$, 
\begin{equation*}
e_{i}x=
\begin{cases}
x &\quad \text{if $h(\alpha)=i$}\\
0 &\quad \text{otherwise}
\end{cases}\qquad,\qquad
xe_{i}=
\begin{cases}
x &\quad \text{if $t(\alpha)=i$}\\
0 &\quad \text{otherwise}
\end{cases}
\end{equation*}
\begin{defin}
\label{def-MQ} Let $\Ml=\bigoplus_{\alpha\in E}\Ml_{\alpha}$. The $\Ml$-twisted quiver algebra of  $Q$ is the tensor algebra of 
$\Ml$ over $\B$:
\begin{equation}
\Ml Q=T_{\B}\Ml\,.
\end{equation}

\end{defin}

In the particular case where $\Ml_{\alpha}=\Ol_{X}\alpha$ for all $\alpha\in E$, 
there is the following isomorphism of $\B$-algebras, which generalises that in eq.~\eqref{tensorvspath}, 
\begin{equation}
T_{\B} (\mbox{$\bigoplus_{\alpha\in E}$}\Ol_X{\alpha}) \simeq \Ol_{X}Q\,,
\end{equation}
where $\Ol_{X}Q= \Ol_X \otimes_{\Z} \Z Q$ is the path algebra of $Q$ over $\Ol_{X}$.

Let $\K\subset\Ml Q$ be a sheaf of two-sided ideals in the $\Ml$-twisted quiver algebra of  $Q$. 
\begin{defin}
The {\em $\Ml$-twisted quiver algebra of $Q$ over $X$ with relations $\Kl$} is the quotient  $\Al=\Ml Q/\Kl$.
\end{defin}
We shall denote by $\pi$ the projection $\Ml Q\to\Al$.

Let  $\V$ be a left $\Ml Q$-module.
If we set 
\begin{equation}\label{Vi} \V_i =  e_i \V \end{equation}
 for each $i \in I$, then $\V_i$ is an $\Ol_X$-module and, since the $e_i$ are mutually orthogonal,
\begin{equation}
\V \simeq \bigoplus_{i\in I} \V_i\,.
\label{eq-V=sum_iV_i}
\end{equation}
For each $\alpha\in E$ the presheaf morphism 
\begin{equation}
\begin{aligned}
\Ml_\alpha(U) \otimes_{\Ol_X(U)} \V_{t(\alpha)}(U) &\longrightarrow \V_{h(\alpha)}(U)\\
m \otimes v & \longmapsto  m \,v 
\end{aligned}
\label{eq-def-rho_alpha}
\end{equation}
induces a morphism $\rho_{\V,\alpha}\colon \Ml_\alpha \otimes_{\Ol_X} \V_{t(\alpha)} \longrightarrow \V_{h(\alpha)}$ of $\Ol_X$-modules.

\begin{defin} A left $\Ml Q$-module $\V$ satisfies condition {\em (TP)} if every point $x\in X$ has an open neighbourhood $U$ such that for all $\alpha\in E$  the natural morphism 
\begin{equation} \label{TPiso}
\Ml_\alpha (U) \otimes_{\Ol_X(U)} \V_{t(\alpha)}(U) \longrightarrow \left(\Ml_\alpha   \otimes_{\Ol_X} \V_{t(\alpha)}\right)(U) 
\end{equation}
is an isomorphism.  
\end{defin}

\begin{rem}\label{TPrem}
Situations where this condition is satisfied are  for instance the following:
\begin{itemize} \setlength\itemsep{0.5em}
\item the $\Ol_X$-modules $\Ml_\alpha$ or the  $\Ol_X$-modules $\V_i$ are locally projective; 
\item $X$ is a scheme, and $\Ml_\alpha$ and $\V_i$ are quasi-coherent  $\Ol_X$-modules.
\end{itemize}
\end{rem}
\begin{lemma} 
\label{lm-Kergamma}
Let $\V$, $\W$ be left $\Ml Q$-modules which satisfy condition {\em (TP)}.
Let 
\begin{equation} \gamma_{\V,\W}\colon \bigoplus_{i\in I} \lHom_{\Ol_X}(\V_i, \W_i) \longrightarrow \bigoplus_{\alpha\in E} \lHom_{ \Ol_X}(\Ml_\alpha\otimes_{\Ol_X} \V_{t(\alpha)}, \W_{h(\alpha)})
\end{equation}
be the morphism locally defined by 
\begin{equation}
\gamma_{\V,\W} \bigl(\{f_i\}_{i\in  I}\bigr) = \bigl\{ f_{h(\alpha)} \circ \rho_{\V,\alpha} -  \rho_{\W,\alpha}\circ(\id_{\Ml_\alpha}\otimes f_{t(\alpha)}) \bigr\}_{\alpha \in E}\,,
\label{eq-def-gamma}
\end{equation}
 where $\{f_i\}_{i\in  I}$ is a section on $U$. Then  there is a natural identification $$\ker \gamma_{\V,\W} = \lHom_{\Ml Q}(\V, \W)\,.$$
\end{lemma}
\begin{proof}
$\lHom_{\Ml Q}(\V, \W)$ can be embedded  into  $\bigoplus_{i\in I} \lHom_{\Ol_X}(\V_i, \W_i)$ as an $\Ol_X$-module. In fact for every   $\Ml Q$-linear  morphism $f \colon \V \to \W$ one has 
$
f(\V_i) \subseteq \W_i
$.
Let $U$ be an open set in which condition (TP) holds 
and  let $g = \oplus_{i\in I} g_i \in \bigoplus_{i\in I}\lHom_{\Ol_X}(\V_i, \W_i)(U)$. Since $\Ml Q$ is generated by $\Ml$,  the morphism $g$ is 
$\Ml Q$-linear if and only if 
\begin{equation}
g(m\,v) = m\, g(v)
\label{eq-g-A-lin}
\end{equation} 
for each open subset $U'\subset U$, each $m\in \Ml(U')$ and each 
$v\in \V(U')$. From the definition of $\Ml$ and from eq.~\eqref{eq-V=sum_iV_i} it follows that
\begin{equation}
m=\oplus_{\alpha\in E}m_{\alpha}\,, \qquad v=\oplus_{i\in I}v_{i}\,,
\label{eq-exp-n-and-v}
\end{equation}
with $m_{\alpha}\in\Ml_{\alpha}(U')$ and $v_{i}\in \V_{i}(U')$. By substituting eq.~\eqref{eq-exp-n-and-v} into eq.~\eqref{eq-g-A-lin} one gets
\begin{align}
0&=   g(m\,v) - m\, g(v) = \bigoplus_{\begin{subarray}{c} \alpha\in E\\ i\in I \end{subarray}}\big(g_{h(\alpha)}( m_{\alpha}\,v_{i}) - m_{\alpha}\, g_{i}(v_{i})\big)\\
&=\bigoplus_{\alpha\in E}\big(g_{h(\alpha)}( m_{\alpha}\,v_{t(\alpha)}) - m_{\alpha}\,g_{t(\alpha)}(v_{t(\alpha)})\big)\\
&= \bigoplus_{\alpha\in E}\big(g_{h(\alpha)}(\rho_{\V,\alpha}(m_{\alpha}\otimes v_{t(\alpha)}) - \rho_{\W,\alpha}(m_{\alpha}\otimes g_{t(\alpha)}(v_{t(\alpha)}))\big) \label{star}  \\
&=\bigoplus_{\alpha\in E}\big[\big(g_{h(\alpha)} \circ \rho_{\V,\alpha} -  \rho_{\W,\alpha}\circ(\id_{\Ml_\alpha}\otimes g_{t(\alpha)}) \big)(m_{\alpha}\otimes v_{t(\alpha)})\big]\,,
\end{align}
where  the equality  \eqref{star} is due to the property (TP). The thesis follows.
\end{proof}

\begin{cor} \label{rem-MQ-NQ-A}  Let $\V$, $\W$ be left $\Al$-modules which  with the induced $\Ml Q$-module structure satisfy condition {\em (TP)}. Then  there is a natural identification $$\ker \gamma_{\V,\W} = \lHom_{\Al}(\V, \W)\,.$$

\end{cor} 

We now prove another property of the morphism $\gamma_{\V,\W}$ (cf.~\eqref{eq-def-gamma}), which will be helpful in the proof of Lemma \ref{lm-elem-ex-seq-V}. Note that, if $\V$ and $\W$ are left $\Ml Q$-modules, and  $\V$ is   also a right $\A$-module, the sheaves
\begin{equation}
\bigoplus_{i\in I} \lHom_{\Ol_X}(\V_i, \W_i) \qquad\text{and}\qquad \bigoplus_{\alpha\in E} \lHom_{ \Ol_X}(\Ml_\alpha\otimes_{\Ol_X} \V_{t(\alpha)}, \W_{h(\alpha)})
\end{equation}
have an induced structure of left $\Al$-modules: in fact, if $i\in I$ and if $U'\subseteq U\subseteq X$ are open subsets, it is enough to set  
\begin{equation}
(af)_i(v_{i})=f_{i}(v_{i}a)\,.
\end{equation}
for every $f_{i}\in\lHom_{\Ol_X}(\V_i, \W_i)(U)$, every $a\in\Al(U)$, and every $v_{i}\in\V_{i}(U')$.
\begin{lemma}
\label{lm-gamma-A-lin}
Let $\V$ and $\W$ be left $\Ml Q$-modules, and assume also that $\V$ is a right $\A$-module.
If condition  {\em (TP)}   is satisfied,   the morphism $\gamma_{\V,\W}$ is $\Al$-linear.
\end{lemma}
\begin{proof}
To make the notation more agile we set $\gamma=\gamma_{\V,\W}$. 
Let $U$ be an open set in which condition (TP) holds. 
Let $f = \oplus_{i\in I} f_i \in \bigoplus_{i\in I}\lHom_{\Ol_X}(\V_i, \W_i)(U)$, let $U'$ be any open subset of $U$, and let $a\in\Al(U')$, $m_{\alpha}\in\Ml_{\alpha}(U')$ and $v_{i}\in\V_{i}(U')$, with $\alpha\in E$ and $i\in I$. Then
\begin{equation}
\begin{aligned}
(\gamma(af))_{\alpha}(m_{\alpha}\otimes v_{t(\alpha)})&=(af)_{h(\alpha)}(m_{\alpha}\,v_{t(\alpha)})- m_{\alpha}\,(af)_{t(\alpha)}(v_{t(\alpha)})=\\
&=f_{h(\alpha)}( m_{\alpha}\,v_{t(\alpha)}a)- m_{\alpha}\, f_{t(\alpha)}(v_{t(\alpha)}a)=\\
&=\big((f_{h(\alpha)}\circ\rho_{\V,\alpha})-\rho_{\W,\alpha}\circ(\id_{\Ml_{\alpha}}\otimes f_{t(\alpha)})\big)(m_{\alpha}\otimes v_{t(\alpha)}a)=\\
&=(\gamma(f))_{\alpha}(m_{\alpha}\otimes v_{t(\alpha)}a)=(a\gamma(f))_{\alpha}(m_{\alpha}\otimes v_{t(\alpha)})\,,
\end{aligned}
\end{equation}
where in the first and in the third step we have used the isomorphism \eqref{TPiso}.
\end{proof}
 \bigskip\section{Representations of a twisted quiver algebra}

\label{SecRep}
Let $\Al$ be an $\Ml$-twisted quiver algebra with relations as above.
We introduce a category $\Rep{\Ml Q}$, whose objects we   call ``$\Ml$-twisted representations of $Q$,'' and contains the category of left $\Al$-modules as a full subcategory.

\begin{defin}\label{def-A-rep}
An object of  $\Rep{\Ml Q}$
is a pair $\big(\{\V_{i}\}_{i\in I},\{\phi_{\alpha}\}_{\alpha\in E}\big)$, where  each $\V_{i}$ is an $\Ol_{X}$-module and each 
\begin{equation}
\phi_{\alpha}\colon\Ml_{\alpha}\otimes_{\Ol_{X}}\V_{t(\alpha)}\longrightarrow\V_{h(\alpha)}
\end{equation}
is a morphism of $\Ol_{X}$-modules.

A morphism
$
f\colon  \big(\{\V_{i}\},\{\phi_{\alpha}\}\big) \longrightarrow 
\big(\{\W_{i}\},\{\psi_{\alpha}\}\big)$
is a collection of morphisms $$\{f_{i}\colon\V_{i}\longrightarrow\W_{i}\}_{i\in I}$$ of $\Ol_{X}$-modules such that the diagram
\begin{equation}
\xymatrix{
\Ml_{\alpha}\otimes_{\Ol_{X}}\V_{t(\alpha)} \ar[r]^-{\phi_{\alpha}} \ar[d]_{\id\otimes f_{t(\alpha)}} & \V_{h(\alpha)} \ar[d]^{f_{h(\alpha)}}\\
\Ml_{\alpha}\otimes_{\Ol_{X}}\W_{t(\alpha)} \ar[r]^-{\psi_{\alpha}}  & \W_{h(\alpha)} 
}
\label{eq-comm-dia-f_i}
\end{equation}
 commutes for each  $\alpha\in E$.
\end{defin}

We need to extend condition (TP) to the category $\Rep{\Ml Q}$ of $\Ml$-twisted representations of $Q$.
\begin{defin} An $\Ml$-twisted representation $\big(\{\V_{i}\},\{\phi_{\alpha}\}\big)$ of $Q$
 satisfies condition {\em (TP)} if every point $x\in X$ has an open neighbourhood $U$ such that for all $\alpha\in E$  the natural morphism \eqref{TPiso}  is an isomorphism.  
\end{defin}

The following theorem
is main result of this section. It generalizes Proposition 2.3  in \cite{GoKi} and clarifies the relation between the category
$\Rep{\Ml Q} $ and the category ${\Al}$-{\bf mod} of left $\Al$-modules. 
\begin{thm}
\label{ThmRQmod}
If  the $\Ol_X$-modules $\Ml_{\alpha}$  are locally projective,  the categories $\Rep{\Ml Q} $ and $\modd{\Ml Q}$
are equivalent.   In particular, $\modd{\Al}$ is a full subcategory of $\Rep{\Ml Q} $.
\end{thm}
\begin{proof}
A functor $F\colon\modd{\Ml Q}\longrightarrow\Rep{\Ml Q}$ is defined as follows: for each left $\Ml Q$-module $\V$  set
\begin{equation}
F(\V)=(\{\V_{i}\}_{i\in I},\{\rho_{\V,\alpha}\}_{\alpha\in E})
\end{equation}
(see eqs.~\eqref{Vi} and \eqref{eq-def-rho_alpha})
  and for each morphism $f\colon\V\longrightarrow\W$ of left $\Ml Q$-modules let
\begin{equation}
F(f)=\{f|_{\V_{i}}\}_{i\in I}\,.
\end{equation}
Note that  by Lemma
 \ref{rem-MQ-NQ-A}  $F(f)$ is indeed a morphism in the category $\Rep{\Ml Q}$.

Since $\Ml Q$ is freely generated by $\Ml$ over $\B$, any  $\B$-algebra homomorphism $\chi\colon \Ml Q\longrightarrow \R$ is uniquely determined
by its restriction $\chi_1 \colon \Ml \longrightarrow \R$. 
We want to take $\R = \lEnd_{\B}(\V')$, where $\V'= \bigoplus_{i\in I}\V_{i}$ is given a left $\B$-module structure by letting $e_j u = \delta_{ij} u$
if $u$ is a section of $\V_i$. 
To define $\chi_1$, if 
 $U$ is a small enough open set in $X$,   for each section $m_{\alpha}\in\Ml_{\alpha}(U)$ and for each open subset $U'\subseteq U$, 
$
\chi_1(m_{\alpha})\colon\V_{t(\alpha)}(U')  \longrightarrow \V_{h(\alpha)}(U')$ is given by
\begin{equation}
\chi_1(m_{\alpha})(v) =  \phi_{\alpha}\big(m_{\alpha}|_{U'}\otimes_{\Ol_{X}(U')}v\big)
\end{equation}
where we have made   use of the isomorphism \eqref{TPiso}, which holds true because the $\Ol_X$-modules $\Ml_\alpha$
are locally projective.

Moreover, we define an $\Ol_X$-algebra morphism 
\begin{equation}\label{eq-rho_phi}
\rho_\phi\colon\Ml Q\to  \lEnd_{\Ol_{X}}(\V')\end{equation}
 by composing $\chi$ with
the inclusion $\lEnd_{\B}(\V') \hookrightarrow   \lEnd_{\Ol_{X}}(\V')$.

Thus, we can map an object $(\V,\phi)$ of $\Rep{\Ml Q}$ to $\V'$, which is a the left $\Ml Q$-module  
via the morphism $\rho_\phi$. 
In this way, we obtain a functor  
\begin{equation}
G\colon\Rep{\Ml Q}\longrightarrow\modd{\Ml Q}\,,
\label{eq-funct-G}
\end{equation}
whose action on morphisms is  
\begin{equation} G(\{f_{i}\})=\oplus_{i\in I}f_{i}
\label{eq-G((f_i))}
\end{equation}
if $\{f_{i}\}\colon (\V,\phi) \longrightarrow 
(\W,\psi)$.
 The commutativity of the diagram \eqref{eq-comm-dia-f_i}, together with Lemma
 \ref{rem-MQ-NQ-A}, imply that $G(\{f_{i}\})$ is $\Ml Q$-linear.
To conclude the proof,  one checks that the functors $F$ and $G$ are quasi-inverse to each other. \end{proof}

\begin{rem} The first statement in 
Theorem  \ref{ThmRQmod} is a slight generalization of 
Proposition 5.1 in  \cite{ACGP},  where the $\Ol_{X}$-modules    $\Ml_{\alpha}$ are supposed to  be locally free   of finite rank.   \end{rem}
 
 \begin{rem} In view of Remark \ref{TPrem}, Theorem \ref{ThmRQmod} can be stated in the following alternative ways, without
 assuming that the $\Ol_X$-modules $\Ml_\alpha$ are locally projective:
 \begin{itemize} \item Let $\modd{\Ml Q}_{\rm lp}$ and $\modd{\Al}_{\rm lp}$ be the full subcategories
 of $\modd{\Ml Q}$ and $\modd{\Al}$, respectively, whose objects are locally projective $\Ol_X$-modules. Analogously, we denote by
  $\Rep{\Ml Q}_{\rm lp}$   the full subcategory of $\Rep{\Ml Q}$ made by objects  $\big(\{\V_{i}\},\{\phi_{\alpha}\}\big)$
where the $\Ol_X$-modules $\V_i$  are locally projective. Then, the categories $\Rep{\Ml Q}_{\rm lp} $ and $\modd{\Ml Q}_{\rm lp}$
are equivalent.   In particular, $\modd{\Al}_{\rm lp}$ is a full subcategory of $\Rep{\Ml Q}_{\rm lp} $.
\item Let $\modd{\Ml Q}_{\rm qc}$ and $\modd{\Al}_{\rm qc}$ be the full subcategories
 of $\modd{\Ml Q}$ and $\modd{\Al}$, respectively, whose objects are quasi-coherent $\Ol_X$-modules. Analogously, we denote by
  $\Rep{\Ml Q}_{\rm qc}$   the full subcategory of $\Rep{\Ml Q}$ made by objects  $\big(\{\V_{i}\},\{\phi_{\alpha}\}\big)$
where the $\Ol_X$-modules $\V_i$  are quasi-coherent. Assume that $X$ is a scheme, and
that the $\Ol_X$-modules $\Ml_\alpha$ are quasi-coherent.
Then, the categories $\Rep{\Ml Q}_{\rm qc} $ and $\modd{\Ml Q}_{\rm qc}$
are equivalent.   In particular, $\modd{\Al}_{\rm qc}$ is a full subcategory of $\Rep{\Ml Q}_{\rm qc} $.
 \end{itemize}
 \end{rem}

\section{The untwisted case}\label{untwisted}
We now study in more detail the case when $$\Ml_{\alpha}=\Ol_{X}\alpha,$$ so  that 
$\Ml=\bigoplus_{\alpha\in E}\Ol_{X}\alpha$ and $\Ml Q\simeq \Ol_{X}Q$. Note that condition (TP) is trivially verified, and also
\begin{equation}
\Ml_{\alpha}\otimes_{\Ol_{X}}\V_{t(\alpha)}\simeq\V_{t(\alpha)}
\end{equation}
for all $\alpha\in E$. This fact simplifies matters. In particular, for each $\alpha\in E$, the morphism $\rho_{\V,\alpha}$ defined in \eqref{eq-def-rho_alpha} simply becomes multiplication by $\alpha$.
Moreover, by Theorem \ref{ThmRQmod} there is an equivalence
$G\colon \Rep{\Ol_X Q} \to \Ol_{X}Q\mbox{-\bf mod}$.
The subcategory of $\Rep{ \Ol_{X}Q}$ equivalent to the subcategory $\modd{\Al}$ of  $\Ol_{X}Q\mbox{-\bf mod}$ can be described in an explicit way, as we now show.
 Let $(\U,\phi)$ be an object in $\Rep{ \Ol_{X}Q}$.  For any path $p$ in the quiver $Q$, we set
\begin{equation}
\phi_{p}=
\begin{cases}
\id_{\U_{i}} &\qquad\text{if}\quad p=e_{i}\\
\phi_{\alpha_{1}}\circ\phi_{\alpha_{2}}\circ\cdots\circ\phi_{\alpha_{n}} &\qquad \text{if}\quad p=\alpha_{1}\alpha_{2}\cdots\alpha_{n}\,.
\end{cases}
\end{equation}
So, if the path $p$ starts at the vertex $i$ and ends at the vertex $j$, one has 
\begin{equation}
\phi_{p}\in\Hom_{\Ol_{X}}(\U_{i},\U_{j})\,.
\end{equation} 
Let $U$ be an open subset  of $X$. 
A  section $\eta \in \Ol_{X}Q)(U)$ can be written as a finite sum
\begin{equation}
\eta=\sum_{p\ \text{path in $Q$}} r_{p}\,p
\end{equation}
with  $r_{p}\in \Ol_{X}(U)$. Thus we set 
\begin{equation}
\eta(\phi)=\sum_{p\ \text{path in $Q$}}r_{p}(\phi_{p}|_{U})\in\bigoplus_{i,j\in I}\Hom_{\Ol_{U}}(\U_{i}|_{U},\U_{j}|_{U})=
\End_{\Ol_{U}}\big((\mbox{$\bigoplus_{i\in I}$}\U_{i})|_{U}\big)
\label{eq-def-eta(phi)}
\end{equation}

\begin{defin}
The category $\Rep{\Al}$ is the  full subcategory of $\Rep{\Ol_X Q} $ whose objects $(\U,\phi)$ satisfy the  condition
\begin{equation}
\eta (\phi)=0\qquad\text{for all sections $\eta$ of $\Kl$}\,.
\label{eq-eta(phi)=0}
\end{equation}
\end{defin}
The main result of this Section is the following theorem.
\begin{thm}
Assume that  $\Ml_{\alpha}=\Ol_{X}\alpha$. There is an equivalence $G_0\colon  \Rep{\Al}\longrightarrow \modd{\Al}$ that fits into the commutative diagram of functors 
\begin{equation}
\xymatrix@C+1em{
\Rep{\Al} \ar[r]^-{G_{0}} \ar@{^{(}->}[d]^{J} & \modd{\Al} \ar@{^{(}->}[d]^{R}\\
\Rep{\Ml Q} \ar[r]^-{G} & \Ol_{X}Q\mbox{-\bf mod}
}
\label{eq-dia-cat-R-rep-mod}
\end{equation}
where the functor $G$ is the equivalence introduced in eq.~\eqref{eq-funct-G}, $J$ is the inclusion of $\Rep{\Al}$ into $\Rep{\Ol_X Q}$, and the functor $R$ is the embedding induced by restriction of scalars.
\end{thm}
\begin{proof}
Since $\Al = \Ol_X Q/\K$,  the functor  $R\colon \modd{\Al} \longrightarrow \Ol_{X}Q\mbox{-\bf mod}$ is a full embedding. 
For each object $(\U,\phi)$ in $\Rep{\Al}$ let
$\rho_{\phi}$
be the homomorphism  introduced in eq.~\eqref{eq-rho_phi}. It is easy to check that, for each section $\eta$ of $\Ol_{X}Q$, one has
$
\rho_{\phi}(\eta)=\eta(\phi)
$
(see eq.~\eqref{eq-def-eta(phi)}). Since $\phi$ fulfills condition \eqref{eq-eta(phi)=0}, one has $\rho_{\phi}\vert_\Kl=0$ and there is an induced homomorphism $\tilde{\rho}_{\phi}\colon\Al\longrightarrow \lEnd_{\Ol_{X}}(\U')$, with 
$\U' =  \bigoplus_{i\in I}\U_{i}$, which fits into the   commutative diagram
\begin{equation}
\xymatrix{
 \Ol_{X} Q \ar[r]^-{\pi} \ar[rd]_(.45){\rho_{\phi}} & \Al \ar[d]^{\tilde{\rho}_{\phi}}\\
 & \lEnd_{\Ol_{X}}(\U')
}\,.
\label{eq-pi-rho-tilde(rho)}
\end{equation}
By giving $\U'$ a structure of left $\Al$-module via the morphism $\tilde\rho_\phi$,
in analogy  with  the functor $G$ (cf.~eq.~\eqref{eq-funct-G}),  one defines
a functor 
 $G_0\colon  \Rep{\Al}\longrightarrow \modd{\Al}$, whose action on morphisms is 
$G_{0}(\{f_{i}\})  =\oplus_{i\in I}f_{i}$ for each  morphism $\{f_{i}\}\colon(\U,\phi)\longrightarrow(\W,\psi)$. 
The functor $G_{0}$ is well defined because 
the morphism $\oplus_{i\in I}f_{i}$ is $\Al$-linear. To prove this claim one can argue as follows: first, 
the commutativity of \eqref{eq-pi-rho-tilde(rho)} implies that  the $\Ol_{X}Q$-module structure induced on $\U'$ (resp. on $\mbox{$\bigoplus_{i\in I}$}\W_{i}$) by restriction of scalars is precisely the same as that provided by $\rho_{\phi}$ (resp. $\rho_{\psi}$); second,  Theorem \ref{ThmRQmod} shows that $\oplus_{i\in I}f_{i}$ is $\Ol_{X} Q$-linear; third,  $R$ is a full embedding, so that 
$\oplus_{i\in I}f_{i}$ is $\Al$-linear.

To prove that $G_0$ is an equivalence we show that it has a quasi-inverse. 
For each $\Al$-module $\V$ let 
\begin{equation}
F_{0}(\V)=\big(\{\V_{i}\}_{i\in I},\{\rho_{\V,\alpha}\}_{\alpha\in E}\big)
\end{equation}
(see eqs.~\eqref{Vi} and \eqref{eq-def-rho_alpha}). 
Now, $F_{0}(\V)$ is an object of $\Rep{\Ol_X Q}$; but for each section 
 $\eta$ of $\Ol_{X}Q$, the morphism $\eta(\{\rho_{\V,\alpha}\})$ (same notation as in eq.~\eqref{eq-def-eta(phi)}) is just the multiplication by $\eta$. Since $\Kl$ is inside the annihilator of $\V$,  eq.~\eqref{eq-eta(phi)=0} is satisfied and $F_{0}(\V)$ is actually an object of $\Rep{\Al}$. For each morphism $f\colon\V\longrightarrow\W$ of left $\Al$-modules let
\begin{equation}
F_{0}(f)=\{f|_{\V_{i}}\}_{i\in I}\,,
\end{equation}
which is a morphism in $\Rep{\Al}$ by Corollary \ref{rem-MQ-NQ-A}. We have therefore defined  
a functor $F_{0}\colon\modd{\Al}\longrightarrow \Rep{\Al}$, and it is not hard to check
that $G_{0}$ and $F_{0}$ are quasi-inverse to each other. 

Finally, the commutativity of \eqref{eq-dia-cat-R-rep-mod} follows at once from the commutativity of \eqref{eq-pi-rho-tilde(rho)}.
\end{proof}

\bigskip \section{A spectral sequence} \label{ss}
In this  and   the next section
the sheaves $\Ml_\alpha$ are  assumed to 
 be 
locally free $\Ol_X$-modules   (for instance,
when they are flat finitely presented $\Ol_X$-modules, and $(X,\Ol_X)$ 
is a locally ringed space, they are indeed locally free of finite rank; this is proved  using \cite[Thm.~7.10]{MatRings} and \cite[Chap.~0, 5.2.7]{EGA-I}).

\begin{thm}
\label{thm-spect-seq}
Let $\V$ and $\W$ be two left $\Al$-modules, and suppose the following  conditions are fulfilled:
\begin{itemize} \setlength\itemsep{.3em}
\item[(K1)]  The $\Ol_{X}$-modules 
$\Al^{p}=\K^p/\K^{p+1}$ are a locally free  for all $p\geq 0$;
\item[(K2)]
$\Kl$ is flat as a right $\Ml Q$-module.
\end{itemize}
Then:
\begin{itemize}
\item[a)] 
there exists a convergent first quadrant spectral sequence 
\begin{equation} \El^{p,q}_{\bullet}(\V,\W)\quad\Rightarrow\quad\lExt_{\Al}^{p+q}(\V,\W)\,,
\end{equation}
whose first page is
\begin{equation}
 \El^{p,q}_{1}(\V,\W)=
\begin{cases}
\bigoplus_{i\in I}\lExt^{q}_{\Ol_{X}}\big((\Al^{p/2}\otimes_{\Al}\V)_{i},\W_{i}\big) &\qquad\text{if $p$ is even}\\[3pt]
\bigoplus_{\alpha\in E}\lExt^{q}_{\Ol_{X}}\big(\Ml_{\alpha}\otimes_{\Ol_{X}}(\Al^{(p-1)/2}\otimes_{\Al}\V)_{t(\alpha)},\W_{h(\alpha)}\big) &\qquad\text{if $p$ is odd}
\end{cases}
\label{eq-loc-E_1^(p,q)}
\end{equation}
\item[b)]
there exists a convergent first quadrant spectral sequence 
\begin{equation}E^{p,q}_{\bullet}(\V,\W)\quad\Rightarrow\quad\Ext_{\Al}^{p+q}(\V,\W)\,,
\end{equation}
whose first page is
\begin{equation}
E^{p,q}_{1}(\V,\W)=
\begin{cases}
\bigoplus_{i\in I}\Ext^{q}_{\Ol_{X}}\big((\Al^{p/2}\otimes_{\Al}\V)_{i},\W_{i}\big) &\qquad\text{if $p$ is even}\\[3pt]
\bigoplus_{\alpha\in E}\Ext^{q}_{\Ol_{X}}\big(\Ml_{\alpha}\otimes_{\Ol_{X}}(\Al^{(p-1)/2}\otimes_{\Al}\V)_{t(\alpha)},\W_{h(\alpha)}\big) &\qquad\text{if $p$ is odd}
\end{cases}
\label{eq-gl-E_1^(p,q)}
\end{equation}
\end{itemize}
\end{thm}
To prove Theorem \ref{thm-spect-seq} we need to establish some preliminary results. 

First, we fix a convention to distinguish the two different structures of left $\Al$-module on the sheaf $\Hl=\lHom_{\Al}(\F,\G)$ when
$\F$ is an $\Al$-bimodule and $\G$ is a left $\Al$-module. Let $U'\subseteq U\subseteq X$ be open sets, and let $f\in\Hl(U)$, $a\in\Al(U)$ and $s\in\F(U')$. We denote by 
$af$ the left multiplication (our usual one) given by 
\begin{equation}
(af)(s)=f(sa)
\label{eq-(L1)}
\end{equation}
and by $a \! \ast\! f$ the left multiplication given by 
\begin{equation}
(a\! \ast\! f)(s)=a\bigl (f(s)\bigr)=f(as)\,.
\label{eq-(L2)}
\end{equation}
We adopt an analogous convention for $\Ml Q$-modules.

Next, we notice that for each $p\geq 0$ there is  a  short exact sequence of $\Ml Q$-bimodules: 
\begin{equation}
\xymatrix{
0 \ar[r] & \Kl^{p+1} \ar[r]^-{\iota_{p}} & \Kl^{p} \ar[r]^-{\pi_{p}} & \Al^{p} \ar[r] & 0
}\,,
\label{eq-ex-seq-p}
\end{equation}
where $\iota_{p}$ is the canonical injection and $\pi_{p}$ is the canonical projection (of course, $\pi_{0}=\pi$). Since both $\Kl^{p}$ and $\Al^{p}$ can be decomposed as in eq.~\eqref{eq-V=sum_iV_i}, we set
\begin{equation}
\iota_{p,i}=\iota_{p}|_{\Kl^{p}_{i}}\,,\qquad\pi_{p,i}=\pi_{p}|_{\Al^{p}_{i}}\,.
\end{equation}
Each  left $\Al$-module $\V$ has a natural   left $\Ml Q$-module structure induced by the restriction of scalars. So we can apply the functor $\lHom_{\Ml Q}(-,\V)$ to eq.~\eqref{eq-ex-seq-p},    getting an injection
\begin{equation}
\xymatrix{
0 \ar[r] & \lHom_{\Ml Q}(\Al^{p},\V) \ar[r]^-{\pi_{p*}} & \lHom_{\Ml Q}(\Kl^{p},\V)
}
\label{eq-pi_(p*)}
\end{equation}
of left $\Ml Q$- and $\ast$-left $\Al$-modules (cf.~eq.~\eqref{eq-(L2)}).
\begin{lemma}
\label{lm-iso-pi_(p*)}
For any left $\Al$-module $\V$, the injection \eqref{eq-pi_(p*)} is an isomorphism of left $\Ml Q$- and $\ast$-left $\Al$-modules.
\end{lemma}
\begin{proof}
We show that for each point $x\in X$ the morphism $\pi_{p*,x}$ is surjective. 
Let $U$ be an open neighbourhood of $x$ and $f\in\lHom_{\Ml Q}(\Kl^{p},\V)(U)$. We observe that
\begin{equation}
f(\Kl^{p+1}|_{U})=0\,.
\label{eq-f(K_x)}
\end{equation}
In fact, for any open subset $U' \subset U$,  every section of $\Kl^{p+1}(U')$ can be written as a finite sum  of products of the form $s=s_{1}s_{2}\cdots s_{p+1}$, with $s_{i}\in\Kl(U')$ for $i=1,\dots,p+1$. But one has 
\begin{equation}
f(s)=s_{1}f(s_{2}\cdots s_{p+1})=\pi (s_{1})f(s_{2}\cdots s_{p+1})=0\,.
\end{equation} 
Thus $f$ induces a morphism $\bar{f}\colon\Al^{p}|_{U}\longrightarrow\V|_{U}$
  such that
\begin{equation}
\pi_{p*}(\bar{f})=f\,.
\end{equation}
This proves that $\pi_{p*}|_{U}$ is surjective, and so is $\pi_{p*,x}$.
\end{proof}
\begin{prop}
\label{lm-elem-ex-seq-V}
Suppose that condition {\rm (K1)} is satisfied.
\begin{itemize}
\item
For all left $\Ml Q$-modules $\V$ and for all $p\geq 0$ the morphism $\gamma_{\Kl^{p},\V}$ is surjective. \\[-15pt]
\item
Every left $\Al$-module $\V$ fits into the following exact sequences of left $\Al$-modules:
\begin{multline}
\xymatrix{
0 \ar[r] & \lHom_{\Al}(\Al^{p},\V) \ar[r]
 & \bigoplus_{i\in I}\lHom_{\Ol_{X}}(\Al^{p}_{i},\V_{i}) \ar[r] &
}\\
\xymatrix{
\ar[rr]^-{\gamma_{\Al^{p},\V}} && \bigoplus_{\alpha\in E}\lHom_{\Ol_{X}}(\Ml_{\alpha}\otimes_{\Ol_{X}}\Al^{p}_{t(\alpha)},\V_{h(\alpha)}) \ar[r]
 & \lHom_{\Al}(\Al^{p+1},\V) \ar[r] & 0
} 
\label{eq-elem-ex-seq-V}
\end{multline}
(this makes sense by Lemma \ref{lm-gamma-A-lin}).
\end{itemize}
\end{prop}
\begin{proof}
Throughout this  proof we write $\otimes$ for $\otimes_{\Ol_{X}}$.
Let $\V$ be a left $\Al$-module;   starting from the exact sequence \eqref{eq-ex-seq-p} it is easy to build up the   commutative diagram of left $\Ml Q$-modules shown in Figure \ref{1},
where we   let
\begin{equation}
\begin{aligned}
\theta_{p,11}=\bigoplus_{i\in I}\lHom_{\Ol_{X}}(\pi_{p,i},\V_i)\,,& \qquad \theta_{p,12}=\bigoplus_{i\in I}\lHom_{\Ol_{X}}(\iota_{p,i},\V_i)\\
\theta_{p,21}=\bigoplus_{\alpha\in E}\lHom_{\Ol_{X}}(\id_{\alpha}\otimes\pi_{p,t(\alpha)},\V_{h(\alpha)})\,,&\qquad \theta_{p,22}=\bigoplus_{\alpha\in E}\lHom_{\Ol_{X}}(\id_{\alpha}\otimes\iota_{p,t(\alpha)},\V_{h(\alpha)})\,.
\end{aligned}
\end{equation}
The surjectivity of $\theta_{p,12}$ and   $\theta_{p,22}$ follows from condition (K1). The exactness of the rows     is a consequence of Lemma \ref{lm-Kergamma} and Corollary  \ref{rem-MQ-NQ-A}. 

The surjectivity of $\gamma_{\Ml Q,\V}$ can be shown by arguing essentially as in \cite[Prop.~3.1]{GoKi}. Then the surjectivity of $\gamma_{\Kl^{p},\V}$ for all $p\geq 0$ can be proved inductively
 by using the diagram itself. This proves the first statement.

By applying the Snake Lemma to the    diagram in Figure 1  one deduces the existence of an exact sequence of left $\Ml Q$-modules of the following form:
\begin{equation}
\label{eq-Hom(A^p,V)-Snake-1}
0  \to  \lHom_{\Al}(\Al^{p},\V) \xrightarrow{a_p}   \lHom_{\Ml Q}(\Kl^{p},\V)  \to \lHom_{\Ml Q}(\Kl^{p+1},\V) \xrightarrow{\partial_p}   \coker\gamma_{\Al,\V}  \to  0
\,.
\end{equation}
It is easy to prove that $a_{p}=\pi_{p*}$ (see eq.~\eqref{eq-pi_(p*)}), so that Lemma \ref{lm-iso-pi_(p*)} implies that $a_{p}$ is an isomorphism. It follows that $\partial_{p}$ is an isomorphism as well. So one gets the  left $\Ml Q$-modules isomorphisms 
\begin{equation}
\coker\gamma_{\Al,\V}\simeq\lHom_{\Ml Q}(\Kl^{p+1},\V)\simeq\lHom_{\Ml Q}(\Al^{p+1},\V)\simeq\lHom_{\Al}(\Al^{p+1},\V)\,,
\label{eq-iso-cokergamma-Hom_A}
\end{equation}
where 
the last step follows from  $\Al$ being  a quotient of $\Ml Q$. The resulting isomorphism
$ \coker\gamma_{\Al,\V}\to \lHom_{\Al}(\Al^{p+1},\V)$
 is   $\Al$-linear.
\end{proof}
Let $\V$ be a left $\Al$-module; by splicing   the exact sequences \eqref{eq-elem-ex-seq-V} for all $p\geq0$, and   using the natural isomorphism $\lHom_{\Al}(\Al,\V)\simeq\V$, one obtains   an exact sequence of left $\Al$-modules  
\begin{equation}
\xymatrix{
0 \ar[r] & \V \ar[r] & \C^{0}(\V) \ar[r]^-{d^{\V}_{0}}  & \C^{1}(\V) \ar[r]^-{d^{\V}_{1}}  & \cdots
}
\label{eq-C(V)}
\end{equation}
where
\begin{equation}
\C^{p}(\V)=
\begin{cases}
\bigoplus_{i\in I}\lHom_{\Ol_{X}}(\Al^{p/2}_{i},\V_{i}) &\qquad\text{if $p$ is even}\\
\bigoplus_{\alpha\in E}\lHom_{\Ol_{X}}(\Ml_{\alpha}\otimes_{\Ol_{X}}\Al^{(p-1)/2}_{t(\alpha)},\V_{h(\alpha)})&\qquad\text{if $p$ is odd}
\end{cases}
\label{eq-C^p(V)}
\end{equation}
\begin{lemma}
\label{lm-C(V)-inj}
Suppose that condition {\rm (K2)} is satisfied, and let $\V$ be a left $\Al$-module which is $\V$ injective as an $\Ol_{X}$-module. Then $\C^{p}(\V)$ is an injective left $\Al$-module for all $p\geq 0$.
\end{lemma}
\begin{proof}
We claim that $\Kl^{p}$ is a flat right $\Ml Q$-module for all $p\geq0$. For $p=0$ this is trivial, and for $p=1$ this is condition (K2). By induction, suppose that $\Kl^{p}$ is flat as a right $\Ml Q$-module for $p=1,\dots,q$, for some $q\geq 1$. This implies that the product $\Kl^{q}\otimes_{\Ml Q}\Kl$ is a flat right $\Ml Q$-module. By \cite[Prop.~4.12]{Lam} one has the following isomorphism of $\Ml Q$-bimodules:
\begin{equation}
\Kl^{q}\otimes_{\Ml Q}\Kl\simeq\Kl^{q+1}\qquad\text{for all}\quad q\geq0
\label{eq-iso-K^p-x-K}
\end{equation}
This implies that $\Kl^{q+1}$ is a flat right $\Ml Q$-module as well. This proves the claim.

Due to the   right $\Al$-module isomorphism
$\Kl^{p}\otimes_{\Ml Q}\Al\simeq\Al^{p}$,
and since flatness is stable under base change, $\Al^{p}$ is a flat right $\Al$-module for all $p\geq0$. By eq.~\eqref{eq-V=sum_iV_i},
 the sheaves $\Al^{p}_{i}$ and $\Ml_{\alpha}\otimes_{\Ol_{X}}\Al^{p}_{t(\alpha)}$ are flat right $\Al$-modules for all $i\in I$, $\alpha\in E$ and $p\geq0$. Moreover eq.~\eqref{eq-V=sum_iV_i} 
 implies that the sheaves $\V_{i}$ are injective $\Ol_{X}$-modules for all $i\in I$.

Let $\F$ be a flat right $\Al$-module; there  is a natural isomorphism of functors
\begin{equation}
\Hom_{\Al}(-,\lHom_{\Ol_{X}}(\F,\V_{i}))\simeq\Hom_{\Ol_{X}}(\F\otimes_{\Al}-,\V_{i})\,.
\end{equation}
Since the functor on the right-hand side is exact, the thesis follows (cf.~\cite[Lemma 3.5]{Lam}).
\end{proof}
Let $\W$ be a left $\Al$-module, so also an $\Ol_{X}$-module by restriction of scalars. We choose a resolution $(\W^{\bullet},\partial_{\bullet})$ of $\W$ by injective $\Ol_{X}$-modules:
\begin{equation}
\xymatrix{
0 \ar[r] & \W \ar[r]
 \ar[r] & \W^{0} \ar[r]^-{\partial_{0}} & \W^{1} \ar[r]^-{\partial_{1}} & \W^{2} \ar[r]^-{\partial_{2}} & \cdots
}
\label{eq-W*}
\end{equation}
By the lifting property of injective resolutions,
every $\W^{q}$ inherits a structure of left $\Al$-module such that the differentials $\partial_{q}$ are $\Al$-linear. In particular one can perform a decomposition as in eq.~\eqref{eq-V=sum_iV_i}, and one has   inclusions $\partial_{q}(\W^{q}_{i})\subseteq\W^{q+1}_{i}$ for all $q\geq 0$. We call $\partial_{q,i}$ the restriction of $\partial_{q}$ to $\W^{q}_{i}$, for all $i\in I$ and $q\geq 0$. By using the exact sequence \eqref{eq-C(V)} associated with each term of the resolution \eqref{eq-W*}, one can define a double complex of left $\Al$-modules $(\C^{\bullet,\bullet}(\W),d,\partial)$ by putting
\begin{align}
\C^{p,q}(\W)&=\C^{p}(\W^{q})\\
d_{p,q}&=d_{p}^{\W^{q}}
\label{eq-C^(p,q)-d_(p,q)}
\\
\partial_{p,q}&=
\begin{cases}
\mbox{$\bigoplus_{i\in I}$}\lHom_{\Ol_{X}}(\Al^{p/2}_{i},\partial_{q,i}) &\qquad \text{if $p$ is even}\\
-\mbox{$\bigoplus_{\alpha\in E}$}\lHom_{\Ol_{X}}(\Ml_{\alpha}\otimes_{\Ol_{X}}\Al^{(p-1)/2}_{t(\alpha)},\partial_{q,h(\alpha)}) &\qquad \text{if $p$ is odd}
\end{cases}
\end{align}
for all $p,q\geq 0$. Let $(\T^{\bullet},D_{\bullet})$ be the total complex associated to $(\C^{\bullet,\bullet}(\W),d,\partial)$:
\begin{equation}
\xymatrix{
\T^{0} \ar[r]^-{D_{0}} & \T^{1} \ar[r]^-{D_{1}} & \T^{2} \ar[r]^-{D_{2}} & \cdots
}
\label{eq-T*}
\end{equation}
\begin{lemma}
The complex $(\T^{\bullet},D_{\bullet})$ is a resolution of $\W$ by injective left $\Al$-modules.
\label{lm-T*-inj-res}
\end{lemma}
\begin{proof}
The double complex $(\C^{\bullet,\bullet}(\W),d,\partial)$ fits into the diagram
\begin{equation}
\xymatrix{
0 \ar[r] & \C^{\bullet}(\W) \ar[r] & \C^{\bullet,\bullet}(\W)
}\,.
\end{equation}
$(\T^{\bullet},D_{\bullet})$ is exact in positive degree, and $\ker D_{0}\simeq \W$.

Finally, Lemma \ref{lm-C(V)-inj} 
 implies that the sheaf $\T^{p}$ is an injective left $\Al$-module for all $p\geq0$.
\end{proof}
We have now all of the ingredients needed to prove Theorem \ref{thm-spect-seq}.
\begin{proof}[Proof of Theorem \ref{thm-spect-seq}.]
To prove the first statement we  apply the functor $\lHom_{\Al}(\V,-)$ to the double complex $(\C^{\bullet,\bullet}(\W),d,\partial)$, getting another double complex, which we call  $(\El^{\bullet,\bullet}_{0}(\V,\W),\hat d,\hat \partial)$. So
\begin{multline}
\El^{p,q}_{0}=\El^{p,q}_{0}(\V,\W)=\lHom_{\Al}(\V,\C^{p,q}(\W))\simeq\\
\simeq
\begin{cases}
\bigoplus_{i\in I}\lHom_{\Ol_{X}}\big((\Al^{p/2}\otimes_{\Al}\V)_{i},\W^{q}_{i}\big) &\qquad\text{if $p$ is even}\\
\bigoplus_{\alpha\in E}\lHom_{\Ol_{X}}\big(\Ml_{\alpha}\otimes_{\Ol_{X}}(\Al^{(p-1)/2}\otimes_{\Al}\V)_{t(\alpha)},\W^{q}_{h(\alpha)}\big) &\qquad\text{if $p$ is odd}
\end{cases}
\label{eq-iso-El_0^(p,q)}
\end{multline}
\begin{equation}
\hat d_{p,q}=\lHom_{\Al}(\V,d_{p,q})
\label{eq-El_0^(p,q)-d}
\end{equation}
\begin{multline}
\hat \partial_{p,q}=\lHom_{\Al}(\V,\partial_{p,q})
\simeq\\
\simeq
\begin{cases}
\bigoplus_{i\in I}
\lHom_{\Ol_{X}}\big((\Al^{p/2}\otimes_{\Al}\V)_{i},\partial_{q,i}\big)
&\qquad\text{if $p$ is even}\\
-\bigoplus_{\alpha\in E}
\lHom_{\Ol_{X}}\big(\Ml_{\alpha}\otimes_{\Ol_{X}}(\Al^{(p-1)/2}\otimes_{\Al}\V)_{t(\alpha)},\partial_{q,h(\alpha)}\big)
&\qquad\text{if $p$ is odd.}
\end{cases}
\label{eq-El_0^(p,q)-partial}
\end{multline}
We have used the isomorphism
 $\Al^{r}_{i}\otimes_{\Al}\V\simeq (\Al^{r}\otimes_{\Al}\V)_{i}$ for all $r\geq0$ and $i\in I$ , and   have applied the Hom-tensor product adjunction.
 Note that the $\Ol_{X}$-modules $(\Al^{r}\otimes_{\Al}\V)_{i}$ and $\W^{q}_{i}$  may fail to be left $\Al$-modules, so that the sheaves $\El^{p,q}_{0}(\V,\W)$ are not left $\Al$-modules in general, and the natural isomorphism used in eq.~\eqref{eq-iso-El_0^(p,q)}  may be $\Ol_{X}$-linear but not $\Al$-linear.

There are two first quadrant spectral sequences 
associated to $(\El^{\bullet,\bullet}_{0}(\V,\W),\hat d,\hat\partial)$,   both  
 abutting to
\begin{align}
\Hg^{n}\left(\lHom_{\Al}(\V,\T^{\bullet}),\widehat D\right),&\qquad\text{where}\qquad \widehat D =\lHom_{\Al}(\V,D)\,.
\end{align}
Lemma \ref{lm-T*-inj-res} implies that
\begin{equation}
\Hg^{n}\left(\lHom_{\Al}(\V,\T^{\bullet}),\widehat D\right)\simeq\lExt^{n}_{\Al}(\V,\W)
\end{equation}
for all $n\geq 0$. Since $(\W^{\bullet},\partial)$ is a resolution of $\W$ by injective $\Ol_{X}$-modules, by taking the cohomology of the double complex $\El^{\bullet,\bullet}_{0}$ with respect to the differential $\hat \partial$  one gets eq.~\eqref{eq-loc-E_1^(p,q)}.

The proof of the second statement is analogous.
\end{proof}

If  $\Kl=0$, so that $\Al=\Ml Q$, Theorem \ref{thm-spect-seq} reduces to \cite[Thm.~4.1]{GoKi}, which we restate here in a sligthly different form.
\begin{cor}
\label{cor-GoKi}
Let $\V$ and $\W$ be two left $\Ml Q$-modules.
\begin{itemize}
\item
There exist two collections of $\Ol_{X}$-modules $\{\El^{q}_{-}(\V,\W)\}_{q=0}^{\infty}$ and $\{\El^{q}_{+}(\V,\W)\}_{q=0}^{\infty}$ that fit into exact sequences 
\begin{multline}
\xymatrix{
0 \ar[r] & \El^{q}_{-}(\V,\W) \ar[r] & \bigoplus_{i\in I}\lExt^{q}_{\Ol_{X}}(\V_{i},\W_{i}) \ar[r] &
}\\
\xymatrix{
\ar[r] & \bigoplus_{\alpha\in E}\lExt^{q}_{\Ol_{X}}\big(\Ml_{\alpha}\otimes_{\Ol_{X}}\V_{t(\alpha)},\W_{h(\alpha)}\big) \ar[r] & \El^{q}_{+}(\V,\W) \ar[r] & 0
}\,,
\label{eq-ex-seq-E_(+-)}
\end{multline}
for all $q\geq0$ and such that
\begin{equation}
\lExt^{q}_{\Ml Q}(\V,\W)\simeq
\begin{cases}
\El^{0}_{-}(\V,\W) &\qquad\text{if}\quad q=0\\[3pt]
\El^{q}_{-}(\V,\W)\oplus\El^{q-1}_{+}(\V,\W) &\qquad\text{if}\quad q>0
\end{cases}\,.
\label{eq-iso-E_(+-)}
\end{equation}
If $q=0$, the middle morphism in the sequence \eqref{eq-ex-seq-E_(+-)} is $\gamma_{\V,\W}$.
\item
There exist two collections of $\Gamma(X,\Ol_{X})$-modules $\{E^{q}_{-}(\V,\W)\}_{q=0}^{\infty}$ and $\{E^{q}_{+}(\V,\W)\}_{q=0}^{\infty}$ that fit into exact sequences  
\begin{multline}
\xymatrix{
0 \ar[r] & E^{q}_{-}(\V,\W) \ar[r] & \bigoplus_{i\in I}\Ext^{q}_{\Ol_{X}}(\V_{i},\W_{i}) \ar[r] &
}\\
\xymatrix{
\ar[r] & \bigoplus_{\alpha\in E}\Ext^{q}_{\Ol_{X}}\big(\Ml_{\alpha}\otimes_{\Ol_{X}}\V_{t(\alpha)},\W_{h(\alpha)}\big) \ar[r] & E^{q}_{+}(\V,\W) \ar[r] & 0
}\,,
\end{multline}
for all $q\geq0$ and such that
\begin{equation}
\Ext^{q}_{\Ml Q}(\V,\W)\simeq
\begin{cases}
E^{0}_{-}(\V,\W) &\qquad\text{if}\quad q=0\\
E^{q}_{-}(\V,\W)\oplus E^{q-1}_{+}(\V,\W) &\qquad\text{if}\quad q>0
\end{cases}
\end{equation}

\end{itemize}
\end{cor}
\begin{proof}
We prove only the first statement as the second  is completely analogous.

As we may assume $\Kl=0$, with the notation of eq.~\eqref{eq-loc-E_1^(p,q)} one has
\begin{equation}
\El^{p,q}_{1}(\V,\W)=
\begin{cases}
\bigoplus_{i\in I}\lExt^{q}_{\Ol_{X}}(\V_{i},\W_{i}) &\qquad\text{if $p=0$}\\
\bigoplus_{\alpha\in E}\lExt^{q}_{\Ol_{X}}\big(\Ml_{\alpha}\otimes_{\Ol_{X}}\V_{t(\alpha)},\W_{h(\alpha)}\big) &\qquad\text{if $p=1$}\\
0 &\qquad\text{if $p>1$}
\end{cases}\,.
\label{eq-loc-E_1^(p,q)-K=0}
\end{equation}
Thus the spectral sequence stabilizes at the second step, with
\begin{equation}
\El^{p,q}_{\infty}\simeq
\begin{cases}
\El^{p,q}_{2} &\qquad \text{if} \qquad p=0,1\,;\\
0 & \qquad \text{if} \qquad p>1
\end{cases}\,.
\label{eq-E^(p,q)_r-r>1-K=0}
\end{equation}
Since the sheaves $
\El^{p,q}_{2}$ fit into exact sequences of the   form
\begin{equation}
\xymatrix{
0 \ar[r] & \El^{0,q}_{2} \ar[r] & \El^{0,q}_{1} \ar[r] & \El^{1,q}_{1} \ar[r] & \El^{1,q}_{2} \ar[r] & 0
}\,,
\end{equation}
eqs.~\eqref{eq-ex-seq-E_(+-)} and \eqref{eq-iso-E_(+-)} follow from Theorem \ref{thm-spect-seq}.

One can check that, up to composition with the natural isomorphisms
$$
\lHom_{\Ml Q}\big(\V,\mbox{$\bigoplus_{i\in I}$}\lHom_{\Ol_{X}}((\Ml Q)_{i},\W_{i})\big)\simeq\mbox{$\bigoplus_{i\in I}$}\lHom_{\Ol_{X}}(\V_{i},\W_{i})
$$
and 
\begin{multline}
\lHom_{\Ml Q}\big(\V,\mbox{$\bigoplus_{\alpha\in E}$}\lHom_{\Ol_{X}}(\Ml_{\alpha}\otimes_{\Ol_{X}}(\Ml Q)_{t(\alpha)},\W_{h(\alpha)})\big)  \simeq   \\
\mbox{$\bigoplus_{\alpha\in E}$}\lHom_{\Ol_{X}}(\Ml_{\alpha}\otimes_{\Ol_{X}}\V_{t(\alpha)},\W_{h(\alpha)}),
\end{multline}
one has
\begin{equation}
\lHom_{\Ml Q}\big(\V,\gamma_{\Ml Q,\W})=\gamma_{\V,\W}\,.
\end{equation}
\end{proof}

The statement  in Theorem \ref{thm-spect-seq} simplifies substantially
whenever the sheaf $\V$ is locally free as an $\Ol_{X}$-module. For any pair $(\V,\W)$ of left $\Al$-modules we introduce the  complex
$\C^{\bullet}(\V,\W) = \El_{1}^{\bullet,0}(\V,\W)$ (see eq.~\eqref{eq-loc-E_1^(p,q)}) with the differentials 
$d^{\V,\W}$ induced by the spectral sequence.
Explicitly one has:
\begin{equation}
\C^{p}(\V,\W)=\begin{cases}
\bigoplus_{i\in I}\lHom_{\Ol_{X}}\big((\Al^{p/2}\otimes_{\Al}\V)_{i},\W_{i}\big) &\qquad\text{if $p$ is even}\\
\bigoplus_{\alpha\in E}\lHom_{\Ol_{X}}\big(\Ml_{\alpha}\otimes_{\Ol_{X}}(\Al^{(p-1)/2}\otimes_{\Al}\V)_{t(\alpha)},\W_{h(\alpha)}\big)&\qquad\text{if $p$ is odd}
\end{cases}
\label{eq-C^p(V,W)}
\end{equation}
\begin{cor}
\label{cor-spect-seq-V-loc-free}
Let $\V$ and $\W$ be two left $\Al$-modules. Suppose that 
conditions {\rm(K1)}--{\rm(K2)} are satisfied, 
and that $\V$ is locally free as an $\Ol_{X}$-module. For all $p\geq0$ there are isomorphisms
\begin{equation}
\lExt^{p}_{\Al}(\V,\W)\simeq\Hg^{p}\big(\C^{\bullet}(\V,\W),d^{\V,\W}_{\bullet}\big)\,.
\end{equation}
\end{cor}
\begin{proof}
We show that the functors
\begin{equation}
\lHom_{\Ol_{X}}\big((\Al^{p}\otimes_{\Al}\V)_{i},-\big)\qquad\text{and}\qquad\lHom_{\Ol_{X}}\big(\Ml_{\alpha}\otimes_{\Ol_{X}}(\Al^{p}\otimes_{\Al}\V)_{t(\alpha)},-\big)
\label{eq-two-exact-funct}
\end{equation}
are exact for all $i\in I$, $\alpha\in E$ and $p\geq0$. 
From the exact sequence \eqref{eq-ex-seq-p},   using condition (K1)  and   arguing  by  induction, one sees that the functor $\lHom_{\Ol_{X}}(\Kl^{p},-)$ is exact for all $p\geq 0$. Hence Corollary \ref{cor-GoKi} implies that for all left $\Ml Q$-modules $\W$ one has
\begin{equation}
\lExt^{i}_{\Ml Q}(\Kl^{p},\W)\simeq
\begin{cases}
\coker\gamma_{\Kl^{p},\W} &\qquad\text{if}\quad i=1\\
0 &\qquad\text{if}\quad i>1
\end{cases}
\end{equation}
for all $p\geq0$. Proposition \ref{lm-elem-ex-seq-V} implies that the functor $\lHom_{\Ml Q}(\Kl^{p},-)$ is exact for all $p\geq0$.

Since $\Al$ is a quotient of $\Ml Q$, 
using   Lemma \ref{lm-iso-pi_(p*)}  one deduces a   natural isomorphism of functors
\begin{equation}
\lHom_{\Al}(\Al^{p},-)\simeq\lHom_{\Ml Q}(\Al^{p},{}_{\Ml Q}(-))\simeq\lHom_{\Ml Q}(\Kl^{p},{}_{\Ml Q}(-))\,,
\end{equation}
where ${}_{\Ml Q}(-)\colon\modd{\Al}\longrightarrow \modd{\Ml Q}$ is the restriction of scalars. Since ${}_{\Ml Q}(-)$ is exact,  the functor $\lHom_{\Al}(\Al^{p},-)$ is exact as well. 
Let $\V$ be a left $\Al$-module; one has an isomorphism  
\begin{equation}
\lHom_{\Ol_{X}}(\Al^{p}\otimes_{\Al}\V,-)\simeq\lHom_{\Al}(\Al^{p},\lHom_{\Ol_{X}}(\V,-))\,.
\end{equation}
It follows that when  $\V$   is locally free as an $\Ol_{X}$-module,  the functor $\lHom_{\Ol_{X}}(\Al^{p}\otimes_{\Al}\V,-)$ is exact. Since $(\Al^{p}\otimes_{\Al}\V)_{i}$ is a direct summand of $\Al^{p}\otimes_{\Al}\V$ and   the sheaves $\Ml_{\alpha}$ are locally free $\Ol_{X}$-modules,    the claim is proved.  

The exactness of the functors \eqref{eq-two-exact-funct} implies that $\El^{p,q}_{1}(\V,\W)= 0$ if $q>0$. One deduces easily that the spectral sequence in eq.~\eqref{eq-loc-E_1^(p,q)} stabilizes at the second step, and that
\begin{equation}
\El^{p,q}_{\infty}=
\begin{cases}
\El^{p,0}_{2} & \text{if $q=0$}\\
0 & \text{if $q>0$}
\end{cases}\qquad\simeq\qquad
\begin{cases}
\Hg^{p}\big(\C^{\bullet}(\V,\W),d^{\V,\W}_{\bullet}\big) &\qquad\text{if $q=0$}\\
0 &\qquad\text{if $q>0$}
\end{cases}
\end{equation}
The thesis follows from Theorem \ref{thm-spect-seq}.
\end{proof}

\begin{ex}
Let $\Bbbk$ be a field, and   $J\subseteq\Bbbk Q$   a two-sided ideal. The  short exact sequence of $\Bbbk Q$-bimodules
\begin{equation}
\xymatrix{
0 \ar[r] & J \ar[r] & \Bbbk Q \ar[r] & \Lambda \ar[r] & 0
}\,.
\label{eq-ex-seq-Lambda}
\end{equation}
defines  a $\Bbbk Q$-algebra structure on $\Lambda$. 
Denote $\Lambda^{p}=J^{p}/J^{p+1}$ for all $p\geq0$. 
We assume that $\Ol_{X}$ is a sheaf of $\Bbbk$-algebras;
tensoring  \eqref{eq-ex-seq-Lambda} by $\Ol_X$ over $\Bbbk$  we obtain 
\begin{equation}
\xymatrix{
0 \ar[r] & \Kl \ar[r] & \Ol_{X} Q \ar[r] & \Al \ar[r] & 0
}\,,
\label{eq-ex-seq-Lambda-tw}
\end{equation}
where   $\Kl=J\otimes_{\Bbbk}\Ol_{X}$ and $\Al=\Lambda\otimes_{\Bbbk}\Ol_{X}$. 
 
The conditions (K1) and (K2) required by Theorem \ref{thm-spect-seq}
  are fulfilled. Indeed, condition (K1) is obvious, as $\Al^{p}=\Lambda^{p}\otimes_{\Bbbk}\Ol_{X}$. 
Concerning the condition (K2), as $\Bbbk Q$ is both a left and a right hereditary ring, $J$ is projective (hence flat) both as a left  and right $\Bbbk Q$-module. {In particular there exists a right $\Bbbk Q$-module $L$  such that there is a  right $\Bbbk Q$-module isomorphism}
\begin{equation}
{J\oplus L\simeq(\Bbbk Q)^{\oplus\Xi}\,,}
\label{eq-J-r-proj}
\end{equation}
{for some   index set $\Xi$ i. 
Tensoring by $\Ol_X$ over $\Bbbk$ one gets a right $\Ol_{X} Q$-module isomorphism}
\begin{equation}
{\Kl\oplus (L\otimes_{\Bbbk} \Ol_{X})\simeq(\Ol_{X}Q)^{\oplus\Xi}\,.}
\label{eq-K-r-proj}
\end{equation}
{so that $\Kl$ is a projective right $\Ol_{X} Q$-module, and this implies condition (K2).}
\label{esempioma}
\end{ex}

When  $X = \operatorname{Spec}{\Bbbk}$ one has:

\begin{prop}
\label{cor-spect-seq-Lambda}
Let $V$ and $W$ be two left $\Al$-modules.  Then  for all $p\geq0$
\begin{equation}
\Ext^{p}_{\Al}(V,W)\simeq\Hg^{p}\big(\C^{\bullet}(V,W),d^{V,W}\big)\,.
\end{equation}
\end{prop}
\begin{proof} One shows that conditions (K1) and (K2) are verified as in Example \ref{esempioma}, so that
Corollary  \ref{cor-spect-seq-V-loc-free} applies.
\end{proof}

\bigskip\section{Hypercohomology}

When $\V$ is a locally free $\Ol_X$-module one is able to compute
the functors $\Ext^{p}_{\Al}(\V,-)$ as the hypercohomology of a complex. This generalizes Theorem 5.1 in \cite{GoKi}.
\begin{thm}
\label{thm-hyper}
Let $\V$ and $\W$ be two left $\Al$-modules. Suppose that the conditions {\rm (K1)} and {\rm (K2)} in Theorem \ref{thm-spect-seq} are satisfied, and suppose that $\V$ is locally free as an $\Ol_{X}$-module. Then one has
\begin{equation}
\Ext^{p}_{\Al}(\V,\W)\simeq\Hb^{p}\big(\C^{\bullet}(\V,\W)
,d^{\V,\W}
\big)
\end{equation}
for all $p\geq0$.
\end{thm}
\begin{proof}

We pick up a resolution $(\W^{\bullet},\partial)$ of $\W$ by injective left $\Al$-modules:
\begin{equation}
\xymatrix{
0 \ar[r] & \W \ar[r]^-{\varepsilon} & \W^{0} \ar[r]^-{\partial_{0}} & \W^{1} \ar[r]^-{\partial_{1}} & \W^{2} \ar[r]^-{\partial_{2}} & \cdots
}
\end{equation}
Since $\Al$ is flat as an $\Ol_{X}$-module, \cite[Cor.~3.6A]{Lam} implies that $(\W^{\bullet},\partial)$ is also a resolution of $\W$ by injective $\Ol_{X}$-modules. Hence one can use $(\W^{\bullet},\partial)$ to build up the spectral sequences in Theorem \ref{thm-spect-seq}.

The exactness of the functors \eqref{eq-two-exact-funct} implies that $\El^{p,q}_{1}(\V,\W)= 0$ if $q>0$ (see eq.~\eqref{eq-loc-E_1^(p,q)}), so that the double complex $\El^{\bullet,\bullet}_{0}(\V,\W)$ introduced in eqs.~\eqref{eq-iso-El_0^(p,q)}, \eqref{eq-El_0^(p,q)-d} and \eqref{eq-El_0^(p,q)-partial} is a resolution of $\C^{\bullet}(\V,\W) = \El_{1}^{\bullet,0}(\V,\W)$:
\begin{equation}
\xymatrix{
0 \ar[r] & \C^{\bullet}(\V,\W) \ar[r] & \El^{\bullet,\bullet}_{0}(\V,\W)\,.
}
\label{eq-C*--->E**}
\end{equation}
Let $\big(\El^{\bullet}_{0}(\V,\W),\Delta)$ be the total complex associated to the double complex $(\El^{\bullet,\bullet}_{0}(\V,\W),\hat d,\hat \partial)$. In particular
\begin{equation}
\El^{n}_{0}(\V,\W)=\bigoplus_{p+q=n}\El^{pq}_{0}(\V,\W)
\end{equation}
The morphism \eqref{eq-C*--->E**} induces a quasi isomorphism
\begin{equation}
\xymatrix{
\C^{\bullet}(\V,\W) \ar[r] & \El^{\bullet}_{0}(\V,\W)
}\,.
\label{eq-C*--->E*}
\end{equation}
Since the sheaves $\Al^{p}$ are flat   right $\Al$-modules, as shown in the proof of Lemma \ref{lm-C(V)-inj},  the  $\Ol_{X}$-modules $\Al^{p}\otimes_{\Al}\V$ are flat   for all $p\geq0$. The sheaves $\Al^{p}\otimes_{\Al}\V$ can be decomposed according to eq.~\eqref{eq-V=sum_iV_i}, 
so that the sheaves $(\Al^{p}\otimes_{\Al}\V)_{i}$ are flat $\Ol_{X}$-modules    for all $p\geq0$ and $i\in I$. So again by the Hom-tensor product adjunction, as at the end of the proof of Lemma \ref{lm-C(V)-inj},
the sheaves $\El^{p,q}_{0}(\V,\W)$ are injective   $\Ol_{X}$-modules for all $p,q\geq0$
(cf.~\cite[Lemma 3.5]{Lam}).  
The sheaves $\El^{n}_{0}(\V,\W)$ are injective   $\Ol_{X}$-modules for all $n\geq0$. Taking global sections in the complex $\El^{\bullet}_{0}(\V,\W)$ and taking cohomology one obtains
\begin{equation}
H^{p}\big(\Gamma\big(\El^{\bullet}_{0}(\V,\W)\big)\big)\simeq\mathbb{R}^{p}\Gamma\big(\C^{\bullet}(\V,\W)\big)=\mathbb{H}^{p}\big(\C^{\bullet}(\V,\W)\big)\,.
\label{eq-H*(C*)1}
\end{equation}
Since
\begin{equation}
\El^{p,q}_{0}(\V,\W)=\lHom_{\Al}(\V,\C^{p,q}(\W))
\end{equation}
(see eq.
\eqref{eq-iso-El_0^(p,q)}), it follows that
\begin{align}
\El^{n}_{0}(\V,\W)&=\lHom_{\Al}\left(\V,\mbox{$\bigoplus_{p+q=n}$}\C^{p,q}(\W)\right)=\lHom_{\Al}(\V,\T^{n})\,, \end{align}
whence \begin{align} \Gamma\big(\El^{n}_{0}(\V,\W)\big)&=\Hom_{\Al}(\V,\T^{n})
\label{eq-Gamma(E*)=Hom}\,,
\end{align}
where $(\T^{\bullet},D_{\bullet})$ is the total complex associated to $(\C^{\bullet,\bullet}(\W),d,\partial)$ (see eq.~\eqref{eq-T*}). Lemma \ref{lm-T*-inj-res} states that $(\T^{\bullet},D_{\bullet})$ is a resolution of $\W$ by injective left $\Al$-modules. Eqs.~\eqref{eq-H*(C*)1} and \eqref{eq-Gamma(E*)=Hom} imply that
\begin{equation}
\mathbb{H}^{p}\big(\C^{\bullet}(\V,\W)\big)\simeq H^{p}\big(\Hom_{\Al}(\V,\T^{\bullet})\big)\simeq \Ext_{\Al}^{p}(\V,\W)\,.
\label{eq-H*(C*)2}
\end{equation}
\end{proof}

\section{An example: ADHM sheaves}\label{adhm}
We consider an example where the quiver $Q$ is the doubled Jordan quiver (see Figure \ref{2}). So we have $\B = \Ol_X$ and $\Ml = \Ml_1\oplus \Ml_2$;
we assume that $\Ol_X$ is a sheaf of $\Bbbk$-algebras for a field $\Bbbk$, and 
that $\Ml_\alpha$ is locally free of finite rank. Moreover, we take for $\Kl$ the ideal generated by the commutators
$[m_1,m_2]$, where $m_\alpha$ is a section of $\Ml_\alpha$. 

\begin{rem} In this case, if the bundles $\Ml_\alpha$ have rank one, the category of left $\Al$-modules   is the category of ADHM sheaves considered by Diaconescu
\cite{Diacu}, with trivial framing (i.e., with zero framing sheaf).
\end{rem}

\begin{prop} The conditions {\em (K1)} and {\em (K2)} are satisfied.
\end{prop}
\begin{proof} The two conditions are local, so we can assume that $\Ml_1$ and $\Ml_2$ are free, $\Ml_\alpha \simeq \Bbbk^{r_\alpha} \otimes \Ol_X$. The algebra $\Ml Q$ can be realized as 
$$ \Ml Q = \Ol_X Q_{r_1,r_2} $$
where $Q_{r_1,r_2} $ is the quiver with one vertex and loops $\alpha_1,\dots,\alpha_{r_1},\beta_1,\dots,\beta_{r_2}$, and the ideal sheaf
$\Kl$ in turn is
$$ \Kl\simeq J \otimes_\Bbbk \Ol_X$$
where $J$ is the two-sided ideal of $\Bbbk Q_{r_1,r_2} $ generated by the $r_1r_2$ commutators $[\alpha_j,\beta_\ell]$.  Then the claim follows from the discussion in Example \ref{esempioma}.
\end{proof}

Let $r_1=r_2=1$,  and assume  that $X$ is a scheme over  an algebraically closed field $\Bbbk$. Now $\Lambda=\Bbbk[\alpha_1,\alpha_2]$, and
$\Al = \Lambda\otimes_\Bbbk \Ol_X$ can be thought of as the structure sheaf of the
product $\mathbb A^2_\Bbbk \times_\Bbbk  X$. We may use a result in \cite{SchVas} to prove, under some hypotheses, a vanishing result for the local Ext sheaves. We take
$$\V = \bigoplus_{i\in I} \tilde V_i\boxtimes  \E_i ,\qquad \W = \prod_{j\in J} \tilde W_j \boxtimes \F_j\,,$$
where $\tilde V_i$ and $ \tilde W_j$ are Artinian coherent sheaves on $\mathbb A^2_\Bbbk$,
$\E_i$ and $\F_i$ are locally free $\Ol_X$-modules of finite rank, and the index sets $I$ and $J$ are arbitrary.

\begin{prop}  There are vanishings
$$\lExt^{p}_{\Al}(\V,\W)=0\quad\mbox{for} \quad p\ge 3\,.$$
Moreover, a Serre duality holds:
$$\lExt^{p}_{\Al}(\V,\W) \simeq \lExt^{2-p}_{\Al}(\W,\V)^\vee \quad\mbox{for} \quad p=0,1,2$$
where $^\vee$ denotes duality with respect to $\Ol_X$.
\end{prop}
\begin{proof} 
Due to the properties of the local Ext sheaves, we may assume $\V = \tilde V \boxtimes \Ol_X$
and $\W = \tilde W \boxtimes \Ol_X$. 
Since the morphism $\Lambda\to\Al$ is flat, by base change for the Ext groups \cite[(3.E)]{Mat} one
has
\begin{eqnarray} \lExt^{p}_{\Al}(\V,\W)  & \simeq & 
 \lExt_{\Al}(\tilde V\otimes_\Lambda \Al, \tilde W\otimes_\Lambda \Al ) \\[3pt]
&\simeq & 
 \Ext^p_\Lambda(\tilde V,\tilde W) \otimes_\Lambda \Al \simeq     \Ext^p_\Lambda(\tilde V,\tilde W)  \otimes_\Bbbk \Ol_X\,.
 \end{eqnarray}
 Since the algebra $\Lambda$ is pre-projective in this case, by \cite[Prop.~3.1]{SchVas}  one has $\Ext^p_\Lambda(\tilde V,\tilde W) =0$
for $p\ge 3$ 
(note that $\tilde V$, $\tilde W$ are finite-dimensional over $\Bbbk$), so that the first claim is proved. 
The second follows from the duality $\Ext^p_\Lambda(\tilde V,\tilde W) \simeq \Ext^{2-p}_\Lambda(\tilde W,\tilde V)^\ast$ in \cite{SchVas} ($^\ast$ denotes duality with respect to $\Bbbk$):
\begin{eqnarray} \lExt^{2-p}_\Al(\W,\V)^\vee  & \simeq & \Ext^{2-p}_\Lambda(\tilde W,\tilde V)^\ast \otimes_\Bbbk \Ol_X
\\[3pt] & \simeq & \Ext^p_\Lambda(\tilde V,\tilde W)  \otimes_\Bbbk \Ol_X \simeq \lExt^p_{\Al}(\V,\W)\,.
 \end{eqnarray}
\end{proof}

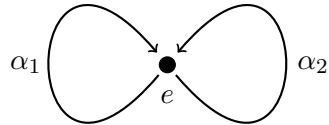
\begin{figure} \begin{center}
\begin{tikzpicture}[scale=6]
 \draw [fill] (0,0) circle (.5pt);
\node [draw=none] {} edge [in=50,out=-50,loop,thick] ();
\node [draw=none] {} edge [in=130,out=230,loop,thick] ();
\node at (0,-0.07) {$e$};
\node at (-0.31,0) {$\alpha_1$};
\node at (0.32,0) {$\alpha_2$};
\end{tikzpicture} 
\end{center} 
\caption{\label{2} The doubled Jordan quiver}
\end{figure}

\bigskip\frenchspacing

\def\cprime{$'$} \def\cprime{$'$} \def\cprime{$'$} \def\cprime{$'$}

%

\newpage
\begin{landscape}
 
 \begin{figure}
 \parbox{5cm}{\hfill}
 
 \vfill
 
\begin{equation} \small  
\xymatrix@C-10pt{
&& 0 \ar[d] & & 0  \ar[d] \\
 0 \ar[r] & \lHom_{\Al}(\Al^{p},\V) \ar[r] 
 & \bigoplus_{i\in I}\lHom_{\Ol_{X}}(\Al^{p}_{i},\V_{i}) \ar[rr]^-{\gamma_{\Al^p,\V}} \ar[d] _{\theta_{p,11}}&&
\bigoplus_{\alpha\in E}\lHom_{\Ol_{X}}(\Ml_{\alpha}\otimes\Al^{p}_{t(\alpha)},\V_{h(\alpha)})\ar[r] \ar[d]^{\theta_{p,21}}&
\coker\gamma_{\A^{p},\V} \ar[r] & 0  \\
 0 \ar[r] & \lHom_{\Ml Q}(\Kl^{p},\V)\ar[r] & \bigoplus_{i\in I}\lHom_{\Ol_{X}}(\Kl^{p}_{i},\V_{i})\ar[rr]^-{\gamma_{\Kl^p,\V}}\ar[d]_{\theta_{p,12}} &&
 \bigoplus_{\alpha\in E}\lHom_{\Ol_{X}}(\Ml_{\alpha}\otimes\Kl^{p}_{t(\alpha)},\V_{h(\alpha)}) \ar[d]^{\theta_{p,22}}\\
 0 \ar[r] &  \lHom_{\Ml Q}(\Kl^{p+1},\V) \ar[r]  & \bigoplus_{i\in I}\lHom_{\Ol_{X}}(\Kl^{p+1}_{i},\V_{i})  
  \ar[rr]^-{\gamma_{\Kl^{p+1},\V}} \ar[d] 
  &&
 \bigoplus_{\alpha\in E}\lHom_{\Ol_{X}}(\Ml_{\alpha}\otimes\Kl^{p+1}_{t(\alpha)},\V_{h(\alpha)})\ar[d]  \\
 && 0 && 0 
}
\end{equation}

\bigskip
\caption{\label{1} }
\vfill
  \parbox{5cm}{\hfill}
  \end{figure}
\end{landscape}

\end{document}